\documentclass[11pt, letterpaper]{amsart}
\usepackage[left=1in,right=1in,bottom=1in,top=1in]{geometry}
\usepackage{bm}
\usepackage{graphicx}
\usepackage{amsmath,amsthm,amssymb}
\usepackage[all,cmtip]{xy}
\usepackage{setspace}
\usepackage{color}
\usepackage[usenames,dvipsnames,svgnames,table]{xcolor}
\usepackage{url}
\usepackage{multirow}
\usepackage{subcaption}

\theoremstyle{plain}

\newtheorem{thm}{Theorem}

\newtheorem {pr}[thm]{Proposition}
\newtheorem {eg}[thm]{Example}
\newcommand{\edit}[1]{\textcolor{red}{EDIT: \texttt{#1}}} 

\usepackage{varioref}
\usepackage{hyperref}
\usepackage{cleveref}

\bmdefine{\ww}{w}
\bmdefine{\xx}{x}
\bmdefine{\yy}{y}
\bmdefine{\zz}{z}
\bmdefine{\rr}{r}
\def\zero{\mathbf{0}} 

\usepackage[numbers]{natbib}

\begin{document}

\title[Multivariate algebraic asymptotics]{Asymptotics of coefficients of algebraic series via embedding into rational series (extended abstract)}


\author{Torin Greenwood
\and
Stephen Melczer
\and
Tiadora Ruza
\and
Mark C. Wilson
}

\newcommand{\Addresses}{{
  \bigskip
  \footnotesize
  T.~Greenwood, \textsc{Department of Mathematics, North Dakota State University}\par\nopagebreak \textit{E-mail address}: \href{mailto:torin.greenwood@ndsu.edu}{torin.greenwood@ndsu.edu}
  
  \medskip
  
  S.~Melczer, \textsc{Department of Combinatorics and Optimization, University of Waterloo}
  
  \medskip
  
  T.~Ruza, \textsc{Department of Combinatorics and Optimization, University of Waterloo}
  
  \medskip
  
  M.C.~Wilson, \textsc{Department of Mathematics \& Statistics, University of Massachusetts Amherst}\par\nopagebreak\textit{E-mail address}: \href{mailto:markwilson@umass.edu}{markwilson@umass.edu}
  
%




%
}}





\begin{abstract} We present a strategy for computing asymptotics of coefficients of $d$-variate algebraic generating functions. Using known constructions, we embed the coefficient array into an array represented by a rational generating functions in $d+1$ variables, and then apply ACSV theory to analyse the latter. This method allows us to give systematic results in the multivariate case, seems more promising than trying to derive analogs of the rational ACSV theory for algebraic GFs, and gives the prospect of further improvements as embedding methods are studied in more detail. 
\end{abstract}
%
%




\maketitle

\section{Introduction}
\label{s:intro}

Algebraic generating functions (GFs) occur throughout combinatorics and its applications. Combinatorial situations that frequently lead to algebraic GFs include: recursive structures described by context-free languages, such as binary trees and constrained (random) walks; RNA secondary structures; applications of the kernel method, the quadratic method for maps, and generalizations. 
Our catalog of test problems \cite{GrMW2021} lists 20 algebraic GFs occurring in the recent research literature, 19 of which are multivariate.

A classic univariate example is the Catalan GF
$$
\frac{1 - \sqrt{1 - 4x}}{2x} = \sum_{n\geq 0} a_n x^n
$$
where $a_n$ can be interpreted as the number of  noncrossing partitions of $\{1, \dots, n\}$, binary trees with $n$ nodes, rooted ordered trees with $n$ edges, Dyck paths of semilength $n$, etc. Asymptotics of coefficients of univariate algebraic GFs, such as this one, can be derived systematically via the local  \emph{singularity analysis} of Flajolet and Odlyzko \cite{FlOd1990} combined with a global analysis of singularities of the GF \cite{FlSe2009, Chab2002}.
Thus the Catalan GF is easily shown to have $n$th coefficient asymptotic to $4^{n}/\sqrt{\pi n^3}$, which in this simple case can be verified  using the explicit formula $a_n = (2n)!/((n+1)!n!)$ and Stirling's approximation.

However, in more than one variable, there is no such general procedure, and there are many situations where we genuinely want asymptotic information about specific coefficients of multivariate GFs. For example, a Schr\"{o}der tree is a rooted plane tree on $n$ leaves where each non-leaf vertex has at least two children.  The GF encoding the number $a_{m, n}$ of Schr\"{o}der trees with $m$ leaves and $n$ vertices satisfies the algebraic equation
$$V(x, y) := \sum_{m, n = 0}^\infty a_{m, n} x^my^n = xy + \frac{y(V(x,y))^2}{1- V (x, y)}.$$
For large {$m$ and $n$}, approximately how many trees with {$m$ leaves and $n$ vertices} are there? 

The most obvious method of attack on such a problem is to try to mimic the custom contour of integration used by Flajolet and Odlyzko. Apart from the PhD thesis work of Greenwood~\cite{Gree2018}, little has been done in this direction. For the results of that substantial work to apply, an explicit formula for the GF in the form $H^{-\beta}$ is required for some $\beta \not\in \mathbb{Z}_{< 0}$.  The dominant singularity of $1/H$ must be a smooth minimal critical point, $H$ must be analytic near the origin, and $H$ must satisfy some partial derivative constraints at the critical point. Creating a general method from this seems rather daunting.  {The full version of this paper will include examples where this method will not apply.}

Multivariate asymptotics are considerably more difficult to derive than in the univariate case, owing to the much greater range of geometries for the singular set of the GF. Even rational functions prove highly nontrivial, in sharp contrast to their algorithmic analysis in the univariate case. Previous work by Bender, Richmond and Gao \cite{BeRi1983, GaRi1992} deals with limit laws for multivariate situations including algebraic GFs, but does not address the detailed asymptotic questions we need here (corresponding in their model to estimating the probability generating function).
However, the long-running ACSV program \cite{PeWi2013} 
has by now made rational multivariate GFs quite manageable in many cases. Our main approach here is to 
embed the array of coefficients of an algebraic GF into a higher-dimensional array represented by a rational function, and then use the existing ACSV theory to derive asymptotics. This idea was suggested in  \cite{RaWi2012a} but has barely been explored since then. 
%
Here we make progress toward a systematic method, by classifying all examples in our recently collated test problem collection \cite{GrMW2021}, and presenting complete results in Section \ref{s:examples} for some of them. Complete details are in the accompanying Sage worksheet, \url{http://acsvproject.com/AlgSage.ipynb}. {To compute asymptotics for the GF $f(\mathbf{x})$:
\begin{enumerate}
    \item {\bf Preprocess.}  Adding finitely many terms to $f$ or making a variable substitution may help attain necessary embedding conditions, and may lead to nicer \emph{combinatorial} embeddings, where all coefficients are non-negative.  (See Section \ref{ss:preprocess}.)
    \item {\bf Embed.} Using Proposition \ref{pr:saf1}, encode the coefficients of $f(\mathbf{x})$ as the elementary diagonal of a rational function $F(Y, \mathbf{x})$ in one additional variable.  This step requires that $f$ is divisible by some variable occurring in $f$ and that the minimal polynomial for $f$ has a non-zero derivative at the origin.  (See Section \ref{ss:embed}.)
        \item {\bf Identify critical points.} If $F(Y, \mathbf{x}) = G(Y, \mathbf{x})/H(Y, \mathbf{x})$, use a system of polynomial equations derived from $H$ to identify critical points that may contribute to the asymptotics.  Because of space, we restrict to the case where we can find an embedding where $1/H(Y, \mathbf{x})$ is combinatorial, allowing us to identify minimal critical points more easily.
        We leave more advanced ACSV issues, such as noncombinatorial embeddings and critical points at infinity, to follow-up work. (See Section \ref{ss:ACSV}.)
        \item {\bf Compute asymptotics.} Once the contributing critical points are identified, asymptotics may be computed algorithmically.  (See our Sage worksheet.)
\end{enumerate}

}

\section{Basic theory}
\label{s:theory}

Let $F(\xx) = \sum_{\rr} a_\rr \xx^\rr$ be a GF in $d$ variables that is algebraic with minimal polynomial $P(Y, \xx)$. 
For the Catalan example above, $P(Y,x) = xY^2 - Y + 1$%
.

\subsection{Embedding using rational GFs}
\label{ss:embed}


\begin{pr}[\cite{Furs1967}]
\label{pr:furst}
Suppose that $f$ is a univariate algebraic power series defined by the polynomial $P(Y,x)$, that $f(0) = 0$, and that $\partial P/\partial Y(0,0) \neq 0$. Then $f$ is the diagonal $\Delta F$ of 
the power series 
$$
F(Y, x):= \frac{Y^2P_Y(Y,Yx)}{P(Y, Yx)}.
$$
\end{pr}
The condition that {$f(0) = 0$} is necessary for the formula, since the constant term of 
$Y^2P_Y(Y,Yx)$ is zero. Also, $\partial P/\partial Y(0,0) \neq 0$ is necessary: it says  that the algebraic function defined by $P$ has a single branch of multiplicity $1$ that passes through the origin. 

\begin{eg}
\label{eg:catalan}
The Catalan GF has nonzero constant term, so there are two obvious ways to apply Proposition \ref{pr:furst}. Subtracting the constant term yields the embedding into 
$$
F_1(Y,x) = \frac{Y\left(1 - 2Y^2x - 2Yx\right)}{1 - (Y^2x + 2Yx + x)}
$$
We could instead multiply by $x$ to obtain the shifted Catalan GF whose $n$th coefficient is $a_{n-1}$. This yields an embedding into
$$
F_2(Y,x) = \frac{Y(1-2Y)}{1-Y-x}.
$$
In other words, $a_n$ is the $(n,n)$ coefficient of $F_1$ and the $(n+1,n+1)$ coefficient of $F_2$.
\end{eg}

There is a substantial lack of uniqueness in the embedding, since adding any bivariate rational GF with zero diagonal will not change the diagonal.  

Proposition \ref{pr:furst} in fact generalizes to a less well known result in an arbitrary number of variables, which also allows for a single branch of higher multiplicity. When $d\geq 3$ there are various notions of diagonal. {The \emph{elementary diagonal} of a GF $F(\mathbf{x})$ is the $(d-1)$-variate GF encoding the coefficients from $F$ where $x_1$ and $x_2$ have matching powers, $\sum_{\mathbf{r}} a_{r_2, r_2, r_3, \ldots, r_{d}} x_2^{r_2} x_3^{r_3} \cdots x_{d}^{r_{d}}$.}

\begin{pr}[\cite{Safo2000}, Lemma {2}]
\label{pr:saf1}
Suppose that $f$ is an algebraic power series given as a branch of
$P(f(\xx), \xx) = 0$, that $f$ is divisible by $x_1$ and that in some neighborhood of $\zero$, there is a factorization $P(Y, \xx) = (Y - f(\xx))^k u(Y, \xx)$ where $u(0,
\zero) \neq 0$ and $k \geq 1$ is an integer.

Then $f$ is the elementary diagonal of the rational function $F$ given by
$$
F(Y, \xx) = 
\frac{Y^2 P_Y(Y,Yx_1, x_2, \dots, x_d)}{k P(Y,Yx_1, x_2, \dots, x_d)}.
$$
\end{pr}

\begin{eg}
\label{eg:narayana}
A bivariate refinement of the Catalan GF is the Narayana GF
$$
G(x,y):= \frac{1}{2x}{\left( 1 - x(y-1)
   - \sqrt{1 - 2x(y + 1) + x^2(y -1)^2} \right)}
$$
which enumerates noncrossing partitions by set size and number of blocks, rooted ordered trees by edges and leaves, Dyck paths by semilength and number of peaks, etc.  {The minimal polynomial $P(Y,x,y) = xY^2 - Y(1-x(y-1)) + 1$ specializes to the Catalan case on setting $y=1$.}

Now $G(x,y) - 1$ satisfies the hypotheses of Proposition \ref{pr:saf1} with $k=1$ (with respect to the variable $x$), and hence embeds into
$$
\frac{\left(1 - (2Y^2x + Yxy + Yx)\right)Y}{\left(1 - \left(Y^2x + Yxy + Yx + xy\right)\right)}.
$$
As expected, if we set $y = 1$, this GF specializes to $F_1$ in Example \ref{eg:catalan}. Similarly, if we instead multiply by $x$ before embedding, we obtain
$$
 \frac{Y \left(1 - 2Y - Yx(y-1)\right)}{1 - x - Y - Y x(y-1)}
$$
which specializes to $F_2$ in Example \ref{eg:catalan}.
\end{eg}

More generally, we must separate the branches of the algebraic function, which requires more work, undertaken for example in \cite{Safo2000}. We intend to consider this case in a companion paper. For now, we concentrate only on cases where the hypotheses of Proposition \ref{pr:saf1} are met, which turns out to be almost all of the examples in \cite{GrMW2021}. 

\subsection{Basics of ACSV}
\label{ss:ACSV}
We give a very brief overview of this (by now standard) material, and refer to \cite{PeWi2013} and \cite{Melc2021} for full details.
Given a rational $d$-variate GF in the form $F = G/H$ with a power series expansion at the origin, we may often derive asymptotics as follows. First represent the coefficient of $\xx^\rr := x_1^{r_1} \cdots x_d^{r_d}$ via the Cauchy Integral Formula with domain of integration a small torus (product of circles) centered at the origin. We then expand this torus via a homothety until we reach a \emph{minimal critical point} (a type of point with algorithmically checkable properties) lying on the boundary of the domain of convergence. Replacing by a local residue integral we may evaluate this using stationary phase methods. In the case where the minimal critical point {$\mathbf{w}$} is a smooth point of the variety $H = 0$,  this leads directly to an asymptotic expansion {for $\mathbf{r} = n \cdot \hat{\mathbf{r}}$ with $\hat{\mathbf{r}} \in (0, \infty)^d$ fixed as $n \to \infty$}
$$
{\left[\mathbf{x}^{\mathbf{r}}\right] F(\mathbf{x}) \sim \mathbf{w}^{-\mathbf{r}}} \sum_{k\geq 0} a_k |\rr|^{(1-d)/2-k}
$$
where the coefficients can be determined algorithmically in terms of derivatives of $G$ and $H$, {and $\mathbf{w}$ and $a_k$ vary with $\hat{\mathbf{r}}$}. The simplest explicit formula in terms of $G$ and $H$ is only for $k = 0$, but the rational GFs we obtain via the embedding procedures above always have $a_0 = 0$ \cite{RaWi2012b}. Computer algebra implementations computing coefficients $a_k$ are built into Sage \cite{Raic}.

Another key observation is that if $1/H$ is \emph{combinatorial} (all its coefficients are nonnegative), then minimal singularities exist and for each minimal critical point $\ww$, the point $(|w_1|, \dots |w_d|)$ is also a minimal critical point. If the series is \emph{aperiodic} (the {subgroup} generated by the support is all of $\mathbb{Z}^d$) then there is a unique minimal critical point. In particular, if $H = 1 - K$ where $K$ is an aperiodic polynomial with nonnegative coefficients, which is the case in many of our examples below, for each direction there is a unique minimal critical point supplying asymptotics in that direction, and it is strictly minimal and lies in the positive orthant.

In the general noncombinatorial case, it can be hard to determine the contributing points.  Minimal points are still desirable because the contour shifting implicit in the above description is still guaranteed to work in the noncombinatorial case. However minimal points need not exist, and algorithmically determining them is considerably harder. Much recent work has been done, using Morse theory, to deal with this case. In this article we discuss only {examples that can be manipulated until a combinatorial embedding is found}, leaving the asymptotic analysis for {noncombinatorial cases for} a future article.

\section{Worked examples}
\label{s:examples}
We categorize the examples in \cite{GrMW2021} according to the type of embedding obtained. We present a single illustrative example in each category in full detail, and give less detail in other cases. Full computations are given in the accompanying Sage worksheet. 

\subsection{Notes on preprocessing}
\label{ss:preprocess}

Assuming that we have an algebraic power series $f$, there are two main hypotheses needed before we can embed using Proposition \ref{pr:saf1}. They are: (H1) $f$ is divisible by some variable occurring in $f$ and (H2) $\partial P/\partial Y$ does not vanish at the origin.

If $f$ does not vanish at the origin, then H1 cannot apply. In the univariate case, we can remedy this by subtracting the constant term or multiplying by the variable, as we showed above for the Catalan example.  More generally, for the first approach we can subtract initial terms of the power series expansion of $f$. The following result gives a simple necessary condition for this to work.  {Let $\xx_j^\circ = (x_1, \ldots, x_{j - 1}, 0, x_{j + 1}, \ldots, x_d)$.}

\begin{pr}
\label{pr:necessary}
{If $f - f_0$ satisfies H1 with respect to the variable {$x_j$}, for some polynomial $f_0$, then {$f(\xx_j^\circ)$} is a polynomial.}
\end{pr}
\begin{proof}
If we substitute {$x_j=0$}, then we need {$f(\xx_j^\circ) - f_0(\xx_j^\circ)$} to be identically zero.
\end{proof}

Note that if we already satisfy H1 but still apply such an \emph{additive substitution}, then the truth value of H2 will not change. To see this, write $\tilde{f} = f - f_0$ and denote its minimal polynomial by $\tilde{P}$. Then $\partial \tilde{P}/\partial Y$ evaluated at the origin is equal to $\partial P/\partial Y$ evaluated at $\xx = 0, Y = f_0(\zero) = 0$. Thus in this case such a substitution is unnecessary if our goal is simply to embed. However as we see in Example \ref{eg:schroeder}, the nature of the rational function that we embed into may change, and this may be useful.

In all dimensions, the second method of \emph{multiplicative substitution} results in something satisfying H1. The Catalan example shows that we may also satisfy H2. However the situation seems quite tricky: 
the ternary tree analog of the shifted Catalan GF is defined by $Y^3 - xY + x^2 = 0$, and this does not satisfy H2. We do not pursue this approach further in the present article, for space reasons and because it seems less promising.

There is a third trick we can use. Assuming we already have $f(\zero) = 0$ but H1 is not satisfied, we can make a \emph{monomial substitution}. For example, in two variables we may have $(x,z) \mapsto (x,xz)$. This ensures that every term is now divisible by $x$, so H1 holds for the transformed GF $\tilde{f}(x,z) = f(x,xz)$. Furthermore there is no change in the truth value of H2, since the minimal polynomial of $\tilde{f}$ is $P(Y, x, xz)$ and this does not change the value of $\partial P /\partial Y$ evaluated at the origin. We do need to keep more careful track of the relation between the indices
in the original power series and the rational function.



\subsection{Embedding is combinatorial after additive substitution}
\label{ss:additive}

In 6 of our 20 catalogued examples, we can immediately embed into a combinatorial GF once we remove the constant term, and in 4 examples a similar approach with more terms works. In addition to the examples shown below, the catalog examples \#8 (Eu; refinement of noncrossing partitions), \#9 (Do\v{s}lic et al.; RNA secondary structures), \#13 (Flajolet and Sedgewick; patterns in trees), \#14 (the Narayana numbers), \#16 (new and old leaves), and \#19 (Bousquet-M\'{e}lou \& Rechnitzer; bar graphs)  lie in this category. 

\subsubsection{A univariate algebraic family \cite[Example 4]{GrMW2021}}
\label{eg:callan}

As a warm-up we deal with a univariate problem.
In \cite{Call2007}, Callan studied the family
\[
G(x): = G_{a,b}(x) = \frac{1 - ax - \sqrt{1-2ax + (a^2 - 4b)x^2}}{2bx^2},
\]
for constants $a$ and $b$.  
Let $f(x) = G(x) - 1$, and observe that $f$ satisfies the minimal polynomial 
\[
P({Y, x}) = x^2b Y^2 + (2bx^2 + ax - 1)Y + (x^2b + ax) .
\]
Since $P_Y(0, 0) \neq 0$, we can use Proposition \ref{pr:saf1} to embed $f$ in the GF
\[
F({Y, x}) = \frac{Y^2 P_Y({Y, xY})}{P({Y, xY})}
= \frac{{\left(1 - \left(2 \, Y^{3} b x^{2} + 2 \, Y^{2} b x^{2} + Y a x\right)\right)} Y}{1 - \left(Y^{3} b x^{2} + 2 \, Y^{2} b x^{2} + Y b x^{2} + Y a x + a x\right)}.
\]
We look at the coefficients of $F$ in the direction $\rr = [1, 1]$ because $[x^n] f(x) = [{Y^nx^n}] F({Y, x})$.  Let $H({Y, x})$ be the denominator of $F({Y, x})$.  Then, critical points satisfy the system of equations
$\left\{H = 0,\ \ xH_x - YH_Y = 0\right\}$.
When $b \neq a^2, a^2/4$, this system yields two points:
\[
{(Y_1, x_1)} = \left(\frac{-\sqrt{b}}{(a - \sqrt{b})(a - 2\sqrt{b})}, -\frac{a \sqrt{b} - b}{b}\right),
{(Y_2, x_2)} = \left(\frac{\sqrt{b}}{(a + \sqrt{b})(a + 2\sqrt{b})}, \frac{a \sqrt{b} + b}{b}\right).
\]

All critical points are smooth because the system of equations $\{H = 0,\ \ xH_x - YH_Y = 0, H_x = 0\}$ has no solutions.  
The coordinates of {$(Y_2, x_2)$} are always positive for any choice of $a, b > 0$, and this point is strictly minimal.
%
%
We now use, for example, \cite[Thm 5.2]{Melc2021} to find an asymptotic expansion for the coefficients of $F({Y, x})$.  
We end up with 
\begin{equation} \label{eq:Example1-4-Asymptotics}
    [{Y^nx^n}] F({Y, x}) = \frac{b^{-3/4}}{2\sqrt{\pi}} n^{-3/2} (a + 2\sqrt{b})^{n + 3/2} + O(n^{-5/2}).
\end{equation}
Note that when $b = a^2$ or $b = a^2/4$, the non-dominant critical point is no longer relevant, but a similar analysis yields the same asymptotic formula.

We compare this to the asymptotic expansions derived using \cite[Cor. 2]{FlOd1990}.  The original GF $G(x)$ has a unique algebraic singularity at $x = \frac{1}{a + 2\sqrt{b}}$ for $a, b > 0$.  Additionally, if $\alpha_1 = \frac{1}{a + 2\sqrt{b}}$ and $\alpha_2 = \frac{1}{a - 2\sqrt{b}}$, then as $x \to \alpha_1$,
\[
\sqrt{1-2ax + (a^2 - 4b)x^2} \sim \sqrt{\alpha_1} \left(1 - \frac{x}{\alpha_1}\right)^{1/2} \sqrt{(\alpha_2 - \alpha_1)(a^2 - 4b)}.
\]
This implies that as $n \to \infty$
\[
[x^n] \sqrt{1-2ax + (a^2 - 4b)x^2} = -\frac{b^{1/4}}{\sqrt{\pi}}n^{-3/2} (a + 2\sqrt{b})^{n - 1/2} + O(n^{-5/2}).
\]
Plugging this expansion into the equation for $G(x)$ gives a matching asymptotic expression to Equation \eqref{eq:Example1-4-Asymptotics}.

\subsubsection{Dissections \cite[Example 10]{GrMW2021}}
\label{eg:dissections}
Drmota \cite[p.376]{Drmo2009} enumerates dissections of polygons (where the polygons have a marked edge) using a bivariate GF $A(x,y)$, where $x$ counts the number of vertices in the polygon, and $y$ counts the total number of edges in the dissection.  
He gives the following minimal polynomial for $A$:
\[
A(x, y) = xy^2(1 + A(x, y))^2 + xy(1 + A(x, y)) \cdot A(x, y).
\]
Since $A$ is divisible by $x$, using Proposition \ref{pr:saf1} immediately we embed into 
$$F({Y, x, y}) = \frac{{\left(1 - \left(2 \, Y^{2} x y^{2} + 2 \, Y^{2} x y + 2 \, Y x y^{2} + Y x y\right) \right)} Y}{1 - \left(Y^{2} x y^{2} + Y^{2} x y + 2 \, Y x y^{2} + Y x y + x y^{2}\right)}.
$$
We note that $[x^{pn} y^{(1-p)n}]A(x, y) = [{Y^{pn}x^{pn}y^{(1-p)n}}] F({Y, x, y})$, and find that there is a single smooth critical point in this direction:
\[
{(Y, x, y)} = {\left(\frac{1 - 2 \, p}{p}, \frac{p(3p-1)^2}{(1-2p)^3}, \frac{(1-2p)^2}{(3p-1)(1-p)}\right)}.
\]
This critical point has positive coordinates when $1/3 < p < 1/2$.  This corresponds to the range of $p$-values where $A(x, y)$ has positive coefficients%
.   It is clear from the denominator of $F$ that in this case, the critical point is strictly minimal.  
Thus
for $1/3 < p < 1/2$:
\[
[x^{pn} y^{(1-p)n}]A(x, y) = \frac{\sqrt{1-p}}{2\pi p^2 \sqrt{3p - 1}} \cdot \frac{1}{n^2} \cdot \left(\frac{(1-p)^{1-p}}{(1-2p)^{2-4p}(3p-1)^{3p-1}}\right)^n + O\left(\frac{1}{n^3}\right).
\]

\if01
\subsubsection{Non-crossing partitions with $k$ visible blocks}
\label{eg:Eu}
Let $b_{n, k}$ be the number of circular non-singleton non-crossing partitions of $[n]$ with $k$ visible blocks. Eu presents \cite[Prop. 63, p.107]{Eu2002}:
\[
B(x, u) := \sum_{n, k \geq 0} b_{n, k} u^k x^n = \frac{1 + x - \sqrt{1 - 2x + x^2 - 4x^2u}}{2x(1 + xu)}.
\]
Using Proposition \ref{pr:saf1}, $B(x, u) - 1$ can be embedded in a GF $F(x, u, Y)$ where $x$ is attached to $Y$.  We search for asymptotics of $[x^{np} u^{n(1-p)}] B(x, u)$.  For circular representations of partitions, every block is visible.  Thus, for a partition of $[n]$ with no singleton blocks, there are between $1$ and $\lfloor n/2\rfloor$ blocks.  This corresponds to $p \in (2/3, 1)$.
In the $(x, u, Y)$-direction  $\rr = (p, (1-p), p)$, there is a single smooth strictly minimal critical point with positive coordinates that leads to the asymptotic expansion
%
\[
[x^{np}u^{n(1-p)}] B(x, u) = \frac{\sqrt{p}}{2\pi (2p - 1)^2 \sqrt{3p - 2}} \cdot \frac{1}{n^2} \cdot \left(\frac{(p-1)^{2p - 2} p^p}{(3p-2)^{3p-2}} \right)^n + O\left(\frac{1}{n^3}\right).
\]

\fi

\subsubsection{Assembly trees \cite[Example 15]{GrMW2021}}
\label{eg:assembly}

B\'{o}na and Vince
\cite{BoVi2013} define the concept of \emph{assembly
tree} of a graph and show that the (exponential) GF for the number of
assembly trees of the complete bipartite graph $K_{rs}$ is 
$$
f(x,y) = \sum_{rs}a_{rs}x^ry^s = 1 - \sqrt{(1-x)^2 + (1-y)^2 - 1}.
$$
Note that $f(0,y) = y$ and $f(x,0) = x$; indeed, the first order  Maclaurin expansion is $x+y$.
We replace $f$ by $\tilde{f}:=f - x - y$, which obviously has no effect on asymptotics of coefficients, and we can embed $\tilde{f}$ via Proposition \ref{pr:saf1} into
$$
F({Y, x, y}) = \frac{{\left(1 - \left(Y x + Y + y \right)\right)} Y}{1 - \left(Y x + x y + Y/2 + y\right)}.
$$
Here, we analyse the direction $\rr = [{p, p, (1-p)}]$ for variables $({Y, x, y})$ and $p \in (0, 1)$, and obtain a pair of smooth critical points, only one of which is strictly minimal with positive coordinates.
Let $\xi := 1 + 4p - 4p^2$.  This leads to the asymptotic formula
\[
\frac{\sqrt{(1 + \xi)\sqrt{\xi} - 2\xi}\left(1 + \xi^{-1/2}\right)}{4\pi p(1-p)} \cdot \frac{1}{n^2} \cdot
\left( \frac{(\sqrt{\xi} - 2p + 3)}{4(1-p)^2} \cdot \left( \frac{1 + (1-2p)\sqrt{\xi}}{4p^2}\right)^p \right)^n + O\left(\frac{1}{n^3}\right).
\]

For example, when $p = 1/2$ we are looking at the diagonal case $K_{rr}$, and obtain
\[
[x^ny^n]f(x, y) = \frac{1}{\pi 2^{9/4}} \cdot \frac{1}{n^2} \cdot (6 + 4\sqrt{2})^n + O\left(\frac{1}{n^3}\right).
\]
This agrees with \cite[Thm 4.11]{BoVi2013} in the exponential rate, but not in the constant --- ours is correct, as we confirmed by numerical checks.

\subsubsection{Schr\"{o}der trees by leaves and vertices \cite[Example 17]{GrMW2021}}
\label{eg:schroeder}
Recall this class from the introduction. Using Proposition \ref{pr:saf1} on the minimal polynomial for $V$ (with respect to the variable $y$) will yield a noncombinatorial embedding.  Instead, we use the embedding $F({Y, x, y})$ for $V(x, y) - xy$.
Checking in the $({Y, x, y})$ direction $\rr = ({p, p, (1-p)})$ for $p \in (1/3, 1/2)$ (corresponding to non-zero coefficients) yields a single strictly minimal critical point with positive coordinates.
Then, we obtain
\[
[x^{np}y^{n(1-p)}]V(x, y) = \frac{\sqrt{3p - 1}}{2\pi p^2\sqrt{1-p}} \cdot \frac{1}{n^2} \cdot \left(\frac{(1-2p)^{4p-2}}{(1-p)^{p-1}(3p-1)^{3p-1}} \right)^n + O\left(\frac{1}{n^3}\right).
\]

\subsection{Embedding is combinatorial after a monomial substitution}
\label{ss:melczer trick}

When the trick of subtracting a polynomial from $f$ does not work, we can move to another level.
In 2 of our 20 examples, a monomial substitution is effective.
\subsubsection{Bivariate generalization of Catalan numbers \cite[Example 7]{GrMW2021}}
\label{sss:Cossali}

The next example comes from \cite{Coss2003}. 
Let $$g(n, m) = \frac{(2n + m)!}{m!n!(n + 1)!}.$$  
with corresponding GF 
\[
L(x, z) := \sum_{q = 0}^\infty \sum_{m = 0}^\infty g(m, q)x^m z^q = \frac{(1 - z) - \sqrt{(1 - z)^2 - 4x}}{2x}.
\]

Then $L(0,0) \neq 0$ and Proposition \ref{pr:necessary} rules out simple additive substitutions, since $L(x,0)$ is the Catalan GF and L'H\^{o}pital's rule shows that $L(0,z)$ is also not a polynomial. Instead we change variables by replacing $z$ by $xz$. Then $L(x, xz) - 1$ satisfies
\[
P = Y^{2} x + Y x z + 2 \, Y x + x z - Y + x,
\]
which leads to the embedding
\[
F({Y, x, z}) = \frac{1 - \left(2 Y^{2} x + Y x z + 2 Y x\right) Y}{1 - \left(Y^{2} x + Y x z + 2  Y x + x z + x\right)},
\]
where $[x^nz^k] L(x, z) = [{Y^{n + k}x^{n+k} z^n}] F({Y, x, z})$.  Rescaling so $n+k$ becomes $n$, we consider  $[x^{pn} z^{(1-p)n}] L(x, z)$ for a constant $p \in (0, 1)$  as $n \to \infty$.  This corresponds to analysing $[{Y^n x^n z^{(1-p)n}}]F({Y, x, z})$, giving us the direction $\rr = ({1, 1, (1-p)})$. 
There is a single smooth strictly minimal critical point with positive coordinates in this direction, yielding
\[
[x^{pn} z^{(1-p)n}] L(x, z) = \frac{\sqrt{1 - p^2}}{2\pi (1-p)p^2} \cdot \frac{1}{n^2} \cdot \left( \frac{1 + p}{1-p} \cdot \frac{(1-p^2)^p}{p^{2p}}\right)^n + O\left(\frac{1}{n^3}\right).
\]

\subsubsection{Bicolored Motzkin Paths \cite[Example 18]{GrMW2021}}

In \cite[Lemma 2.1]{Eliz2021}, Elizalde derives the GF for bicolored Motzkin paths.  Here, such a path starts and ends at the $x$-axis, never passes below the $x$-axis, and takes steps $U = (1, 1), D = (1, -1)$, and two (colored) types of horizontal steps, $H_1 = (1, 0)$ and $H_2 = (1. 0)$.  Let $a_{m, n}$ be the number of such paths with $m$ total $U$ or $H_1$ steps and $n$ total $D$ or $H_2$ steps.  Then
\[
M(x, y) := \sum_{m, n = 0}^\infty a_{m, n} x^m y^n = \frac{1 - x - y - \sqrt{(1 - x- y)^2 - 4xy}}{2xy}.
\]
Using Proposition \ref{pr:saf1} on $M(x, xy) - 1$ gives the asymptotic formula
\[
[x^{pn} y^{(1-p)n}] M(x, y) = \frac{1}{2\pi (1-p)^2p^2} \cdot \frac{1}{n^2} \cdot \left( \frac{(1-p)^{2p - 2}}{p^{2p}}\right)^n + O\left(\frac{1}{n^3}\right).
\]

\if01
\subsubsection{General quadratic GFs}
Suppose that $F(\xx)$ is algebraic with minimal polynomial of the form $P(Y,\xx) = a(\xx) Y^2 + b(\xx)Y + c(\xx)$. A necessary condition for the hypotheses of 
Proposition \ref{pr:saf1} to hold is $b(\zero) \neq 0$.\edit{check}. 
\fi
\if01
\subsubsection{Leaves in trees}
Recall the Catalan and Narayana examples from the introduction. 
There is a further refinement \cite{ChDE2006} that considers different
types of leaves in a rooted ordered tree.  
Call a leaf of such a  tree \emph{old} if it is the leftmost
child of its parent, and \emph{young} otherwise.  They enumerate such trees
according to the number of old leaves, number of young leaves and number of
edges, finding the algebraic equation
$$
G(x,y,z) = 1 + \frac{z(G(x,y,z) - 1 + x)}{1 - z(G(x,y,z) - 1 + y)}.
$$
Thus $G$ is defined by the polynomial
$$
P(Y,x,y,z) = ***
$$
and we obtain the embedding 
$$
***
$$
\fi

\if01
\subsection{More difficult examples}

So far we have dealt with *** out of *** examples in \cite{GrMW2021}. Of the remaining ones in the catalog \cite{GrMW2021}, Proposition \ref{pr:saf1} applies for almost all \edit{how many}, but we have not yet found a combinatorial embedding in those examples. For example, bilateral Schr\"{o}der paths (see \url{https://oeis.org/A063007}) from $(0,0)$ to $(n,n)$ having $k$ North steps are enumerated by
$$G(t, z) = \frac{1}{\sqrt{1-2z-4tz+z^2}.}$$
The embedding $F({Y, t, z})$ found by using Proposition \ref{pr:saf1} and attaching $z$ to $Y$ yields no affine critical points in any relevant direction $\rr = [{(1-p), p, (1-p)}]$ for $p \in (0, 1)$.

For only *** of them does H2 fail to hold. Such cases can be treated by 
Safonov's algorithm \cite{Safo2000}, which we intend to do in work currently in preparation. 
\fi

\section{Further discussion}
\label{s:discuss}
We believe we have shown that our strategy for computing asymptotics of multivariate algebraic series is a good one. 
Further developments in ACSV theory, some already near completion, will allow analysis of noncombinatorial embeddings. We intend to take this up in a future work, and some cases will be challenging. For example, bilateral Schr\"{o}der paths 
from $(0,0)$ to $(n,n)$ having $k$ North steps are enumerated by
$$G(t, z) = \frac{1}{\sqrt{1-2z-4tz+z^2}.}$$
The embedding $F({Y, t, z})$ found by using Proposition \ref{pr:saf1} and attaching $z$ to $Y$ yields no affine critical points in any relevant direction.


We note further that Furstenberg and Safonov are not the only relevant authors for embedding procedures. For example, Denef \& Lipshitz \cite{DeLi1987} give a somewhat more complicated method, which when applied to the example $x/\sqrt{1-x}$ yields, instead of the rather difficult to analyse embedding given by Proposition \ref{pr:furst}, the simple GF $2xY/(1-x-Y)$. This indicates a brighter future for the general program of computing multvariate asymptotics by reducing the 
algebraic to the rational case.

\section{Acknowledgements}
This material is based upon work supported by the National Science Foundation under Grant Numbers 1641020 and 1916439, an NSERC Discovery Grant, an NSERC USRA and the American Mathematical Society Mathematics Research Community, Combinatorial Applications of Computational Geometry and Algebraic Topology.


\bibliographystyle{plain}

\bibliography{mvgf.bib}

\begin{thebibliography}{10}

\bibitem{Raic}
Sagemath documentation: Asymptotics of multivariate generating series, 2021.

\bibitem{BeRi1983}
Edward~A. Bender and L.~Bruce Richmond.
\newblock Central and local limit theorems applied to asymptotic enumeration.
  {II}. {M}ultivariate generating functions.
\newblock {\em J. Combin. Theory Ser. A}, 34(3):255--265, 1983.

\bibitem{BoVi2013}
Mikl{\'o}s B{\'o}na and Andrew Vince.
\newblock The number of ways to assemble a graph.
\newblock In {\em 2013 Proceedings of the Tenth Workshop on Analytic
  Algorithmics and Combinatorics (ANALCO)}, pages 8--17. SIAM, 2013.

\bibitem{Call2007}
David Callan.
\newblock On generating functions involving the square root of a quadratic
  polynomial.
\newblock {\em Journal of Integer Sequences}, 10(2):3, 2007.

\bibitem{Chab2002}
C.~Chabaud.
\newblock {\em S{\'e}ries g{\'e}n{\'e}ratrices alg{\'e}briques:asymptotique et
  applications combinatoires}.
\newblock PhD thesis, Universit{\'e} {P}aris {VI}, 2002.

\bibitem{Coss2003}
GE~Cossali.
\newblock A common generating function for catalan numbers and other integer
  sequences.
\newblock {\em Journal of Integer Sequences}, 6(2):3, 2003.

\bibitem{DeLi1987}
J.~Denef and L.~Lipshitz.
\newblock Algebraic power series and diagonals.
\newblock {\em J. Number Theory}, 26(1):46--67, 1987.

\bibitem{Drmo2009}
Michael Drmota.
\newblock {\em Random trees: an interplay between combinatorics and
  probability}.
\newblock Springer Science \& Business Media, 2009.

\bibitem{Eliz2021}
Sergi Elizalde.
\newblock The degree of symmetry of lattice paths.
\newblock {\em Ann. Combin.}, pages 1--35, 2021.

\bibitem{FlOd1990}
Philippe Flajolet and Andrew Odlyzko.
\newblock Singularity analysis of generating functions.
\newblock {\em SIAM J. Discrete Math.}, 3(2):216--240, 1990.

\bibitem{FlSe2009}
Philippe Flajolet and Robert Sedgewick.
\newblock {\em Analytic combinatorics}.
\newblock Cambridge University Press, 2009.

\bibitem{Furs1967}
Harry Furstenberg.
\newblock Algebraic functions over finite fields.
\newblock {\em J. Algebra}, 7:271--277, 1967.

\bibitem{GaRi1992}
Zhicheng Gao and L.~Bruce Richmond.
\newblock Central and local limit theorems applied to asymptotic enumeration.
  {IV}. {M}ultivariate generating functions.
\newblock {\em J. Comput. Appl. Math.}, 41(1-2):177--186, 1992.
\newblock Asymptotic methods in analysis and combinatorics.

\bibitem{Gree2018}
Torin Greenwood.
\newblock Asymptotics of bivariate analytic functions with algebraic
  singularities.
\newblock {\em Journal of Combinatorial Theory, Series A}, 153:1--30, 2018.

\bibitem{GrMW2021}
Torin Greenwood, Stephen Melczer, Tiadora Ruza, and Mark~C. Wilson.
\newblock Algebraic multivariate generating functions: test problems.
\newblock 2021.

\bibitem{Melc2021}
Stephen Melczer.
\newblock {\em An Invitation to Analytic Combinatorics: From One to Several
  Variables}.
\newblock Springer Nature, 2021.

\bibitem{PeWi2013}
Robin Pemantle and Mark~C. Wilson.
\newblock {\em Analytic Combinatorics in Several Variables}.
\newblock Cambridge University Press, New York, 2013.

\bibitem{RaWi2012a}
Alexander Raichev and Mark~C. Wilson.
\newblock {A new approach to asymptotics of Maclaurin coefficients of algebraic
  functions}.
\newblock {\em arXiv.org}, 1202.3826, February 2012.

\bibitem{RaWi2012b}
Alexander Raichev and Mark~C. Wilson.
\newblock Asymptotics of coefficients of multivariate generating functions:
  improvements for multiple points.
\newblock {\em Online Journal of Analytic Combinatorics}, 7:25pp, 2012.

\bibitem{Safo2000}
K.~V. Safonov.
\newblock On power series of algebraic and rational functions in {${\bf C}\sp
  n$}.
\newblock {\em J. Math. Anal. Appl.}, 243(2):261--277, 2000.

\end{thebibliography}

\Addresses 
%
%

\end{document}